\documentclass[reqno]{amsart}
\usepackage{amsfonts}
\usepackage{amsmath,amssymb,amsxtra}
\usepackage{float}
\usepackage[colorlinks,linkcolor=RoyalBlue,anchorcolor=Periwinkle, citecolor=Orange,urlcolor=Green]{hyperref}
\usepackage[usenames,dvipsnames]{xcolor}
\usepackage{enumitem}
\usepackage{geometry} 
\usepackage{graphicx}
\usepackage{subfigure}
\usepackage{tikz}\usetikzlibrary{matrix}
\usepackage{url}
\usepackage[all]{xy}
\usepackage{yhmath}
\usepackage{booktabs,bookmark}
\usepackage{amsmath,amsthm,amssymb,mathrsfs}

\numberwithin{figure}{section}

\newtheorem{theorem}{Theorem}[section]
\newtheorem{lemma}[theorem]{Lemma}
\newtheorem{corollary}[theorem]{Corollary}
\newtheorem{main theorem}[theorem]{Main Theorem}

\newtheorem{definition}[theorem]{Definition}
\newtheorem{problem}[theorem]{Problem}
\newtheorem{example}[theorem]{Example}

\newtheorem{remark}{Remark}[section]
\usetikzlibrary{arrows}

\numberwithin{equation}{section}

\newcommand{\Lie}{\mathcal{L}}
\begin{document}
%=========================================================

\title{The edge-girth-regularity of Wenger graphs}

\author{Fuyuan Yang$^{1,2}$}
%\address{FY: $^1$Department of Mathematics, School of Mathematics and Statistics, Guizhou University, 550025, Guiyang, China. 
%\newline{\quad$^2$Department of Public Basic Education, Moutai Institute, 564500, Renhuai, China.}}
%\email{fyyang9@163.com}

%
\author{Qiang Sun$^3$}
%\address{QS: School of Mathematical Science, Yangzhou University, 225009, Yangzhou, China.}

%\email{qsun1987@163.com}
%
\author{Chao Zhang$^{1,*}$}
%\address{CZ: Department of Mathematics, School of Mathematics and Statistics, Guizhou University, 550025, Guiyang, China.}
%\email{zhangc@amss.ac.cn}

%\subjclass[2010]{05C25; 05C35; 05E15}
%\keywords{Wenger graph, edge-girth-regular graphs, extremal edge-girth-regular graphs, generalized Tur\'an number}
\thanks{* The corresponding author.}

\maketitle
\begin{center}	
$^1$Department of Mathematics, School of Mathematics and Statistics, Guizhou University, 

550025, Guiyang, China

$^{2}$Department of Public Basic Education, Moutai Institute, 564500, Renhuai, China

$^3$School of Mathematical Science, Yangzhou University, 225009, Yangzhou, China

	\bigskip
	{ E-mails: fyyang9@163.com; qsun1987@163.com; zhangc@amss.ac.cn }
\end{center}
%=========================================================

\begin{abstract}
%An edge-girth-regular graph $egr(v,k,g,\lambda)$, is a $k$-regular graph of order $v$, girth $g$ and with the property that each of its edges is contained in exactly $\lambda$ distinct $g$-cycles. An $egr(v,k,g,\lambda)$ is called extremal for the triple $(k,g,\lambda)$ if $v$ is the smallest order of any $egr(v,k,g,\lambda)$.
Let $n\ge 1$ be an integer and $\mathbb{F}_q$ be a finite field of characteristic $p$ with $q$ elements.
In this paper, it is proved that the Wenger graph $W_n(q)$ and linearized Wenger graph $L_m(q)$ are edge-girth-regular $(v,k,g,\lambda)$-graphs, and the parameter $\lambda$ of graphs $W_n(q)$ and $L_m(q)$ is completely determined. Here, an edge-girth-regular graph $egr(v,k,g,\lambda)$ means a $k$-regular graph of order $v$ and
girth $g$ satisfying that any edge is contained in $\lambda$ distinct $g$-cycles. As a direct corollary, we obtain the number of girth cycles of graph $W_n(q)$, and
the lower bounds on the generalized Tur\'an numbers $ex(n, C_{6}, \mathscr{C}_{5})$ and $ex(n, C_{8}, \mathscr{C}_{7})$, where $C_k$ is the cycle of length $k$ and  $\mathscr{C}_k = \{C_3, C_4, \dots , C_k\}$.
%\mathscr{C}_{7})$the generalized Tur\'an number $ex(2q^3, C_{8}, \mathscr{C}_{7}) \ge \frac{q^4(q-1)^3(q-2)}{8}$,
Moreover, there exist a family of $egr(2q^3,q,8,(q-1)^3(q-2))$-graphs for $q$ odd,
and the order of graph $W_2(q)$ and extremal $egr(v,q,8,(q-1)^3(q-2))$-graph have same asymptotic order for $q$ odd.
\end{abstract}


%=========================================================
\section{Introduction}
%=========================================================
 All the notions of graph theory in this paper can be found in \cite{BM76,Bi92}. The construction of algebraically defined graph $B\Gamma_n(R;f_2,\dots,f_n)$ is motivated by the problems from extremal graph theory (see \cite{LFU93,LFU95,LFU99,LWA01,LM02,TAT12} for details). Some important subfamilies of algebraically defined graphs have been widely studied. The Wenger graph $W_n(q)$ is an important graph family used to construct graphs on extremal problems about forbidden small even cycles, while the linearized Wenger graph $L_m(q)$ is concerned in spectral graph theory.
 Wenger \cite{W91} introduced a family of $p$-regular bipartite graphs $H_k(p)$, which provided a new example for the lower bound of Tur$\acute{\text{a}}$n number the cycle $C_{2k}$, where $k=2,3,5$. Several years later, a family of $q$-regular bipartite graphs was explicitly constructed by Lazebnik and Viglione in \cite{LFV02}, which is called Wenger graph $W_n(q)=B\Gamma_{n+1}(\mathbb{F}_q;p_1l_1,\dots,p_1l_i,\dots,p_1l_{n})$ for the reason that $H_k(p)$ is isomorphic to $W_k(p)$ for all $k \ge 1$ and $p$ prime.
Another important class of algebraically defined graphs is Lie graphs over finite fields induced by generalized Kac-Moody algebras \cite{TAT12}.
These Lie graphs are closely related to the research of geometry objects \cite{WF64,LFU93, LFU94, LFS17,VAU03,TAT12,AJW10}.
%For example, the generalized $k$-polygons of order $(q, q)$ are classified into three cases: projective plane, generalized quadrangle and generalized hexagon of order $q$ \cite{WF64}.
%The incidence structure of graphs correspond to the relations of Lie operations in Lie algebras \cite{LFU94,TAT12,VAU03,AJW10}.
For example, Lie graphs $\Lie(M_1,3,q)$, $\Lie(M_2,4,q)$ and $\Lie(M_3,6,q)$ over finite fields induced by finite dimensional generalized Kac-Moody algebras are precisely the affine part of the projective plane, generalized quadrangle and generalized hexagon, respectively.
%The graphs $D(k, q)$ was motivated by attempts to generalize the notion of the biaffine part of a generalized polygon \cite{LFU93,LFS17}, and its connected components $CD(k, q)$ are denser than the well-known graphs. The relationship between these graphs is that $W_1(q)$, $\Lie(M_1,3,q)$ and $D(2,q)$ are isomorphic to each other, and $W_2(q)$, $\Lie(M_2,6,q)$ and $D(3,q)$ are isomorphic to each other.

%The spectrum of graphs $W_n(q)$ and $D(4,q)$ are completely described in \cite{SMC14} and \cite{MSW17}, respectively. Recent results show that $D(5,q)$ is nearly Ramanujan graph for $q$ odd \cite{GT22}.
The girth cycles have significant applications in variant problems from extremal
graph theory, finite geometry, coding theory etc., and determining all the girth cycle of graphs is a classical problem in graph theory \cite{C2021,GGMV20,SU09,SW17,T1983}.
The existence of some cycles with certain lengths in Wenger graph $W_n(q)$ and linearized Wenger graph $L_m(q)$ is studied in \cite{LTW18} and \cite{WY17}, respectively. Moreover, Wenger graph $W_n(q)$ is $q$-regular semisymmetric (i.e., regular, edge-transitive and non-vertex-transitive) graph \cite{LFV02}, for any $n \ge 3$ and $q \ge 3$, or $n = 2$ and $q$ odd, and $\Lie(M_3,6,q)$ is $q$-regular semisymmetric graph \cite{YSZ22}, for $q$ is not powers of $2$ or $3$.
More problems and results appear in \cite{LWA01,LFS17} and the references therein.

Inspired by the Moore graphs with even girth, Jajcay et al. \cite{JKM18} introduced a new regularity of graphs called edge-girth-regularity. An {\it edge-girth-regular graph $egr(v,k,g,\lambda)$} is a $k$-regular graph of order $v$ and girth $g$ in which every edge is contained in $\lambda$ distinct $g$-cycles. This concept is a generalization of the well-known concept of $(v, k,\lambda)$-edge-regular graphs, which count the number of triangles. They constructed infinite edge-girth-regular graph families with certain parameters $(k,g,\lambda)$, for example, there exist infinitely many $egr(v,k,g,2)$-graphs for every $k\ge3$ and $g \ge6$. The smallest order of an $egr(v,k,g,\lambda)$-graph fixing the triplet $(k,g,\lambda)$ is called {\it extremal edge-girth-regular graph} in \cite{DFJR21}. Two families of edge-girth-regular graphs with girth $5$ or $6$ are introduced,  and the following conjecture is raised \cite{GAP22}, which is still open until now.

\textbf{ Conjecture 1.} \label{conj 1}For $q \ge 3$ a prime power and $g \in \{8, 12\}$ there exists a family of $egr(2q^{\frac{g-2}{2}},q,g,(q-1)^{\frac{g-2}{2}}(q-2))$-graphs. These graphs are extremal edge-girth-regular graphs.

The present paper mainly considers the edge-girth-regularity of $W_n(q)$ and $L_m(q)$ (compare section 5.2 in \cite{LWA01}). It is proved that $W_n(q)$ and $L_m(q)$ are $egr(v,k,g,\lambda)$-graph,  and the number of its $g$-cycles passing through any edge in $W_n(q)$ and $L_m(q)$ is determined  respectively in the following two theorems.
It is worth mentioning that Theorem 1 partially answers the above conjecture when $g = 8$ and $q$ odd, i.e., there exists a family of $egr(2q^{\frac{g-2}{2}},q,g,(q-1)^{\frac{g-2}{2}}(q-2))$-graphs for $g = 8$ and $q$ odd.

%\textbf{Theorem 1.} The Wenger graph $W_n(q)$ is an $egr(2q^{n+1}, q,8,\lambda_n)$-graph for $n\ge 2$, where $\lambda_2 = (q-1)^3(q-2)$ for $q$ odd, $\lambda_2 = (q-1)^3(q-3)+2(q-1)^2$ for $q$ even, and $\lambda_n = (q-1)^3$ for $n\ge 3$. Moreover, $W_1(q)$ is an $egr(2q^2, q,6,(q-1)^2(q-2))$-graph for $q\ge 3$, $W_1(2)$ is an $egr(8, 2,8,1)$-graph.

\textbf{Theorem 1.} Let $q$ be a prime power and $n$ be an positive integer. Then Wenger graph $W_n(q)$ is an $egr(2q^{n+1}, q,g_n,\lambda_n)$-graph for all $n$ and $q$. More precisely,
\begin{enumerate}
\item $W_1(q)$ is an $egr(2q^2, q,6,(q-1)^2(q-2))$-graph for $q\ge 3$, and $W_1(2)$ is an $egr(8, 2,8,1)$-graph;
\item the girth of $W_2(q)$ is $8$ for all $q$, and $\lambda_2 = (q-1)^3(q-2)$ for $q$ odd, $\lambda_2 = (q-1)^3(q-3)+2(q-1)^2$ for $q$ even;
\item if $n\geq 3$, then the girth $g_n$ of $W_n(q)$ is $8$ and $\lambda_n = (q-1)^3$ .
\end{enumerate}

Given a graph $H$ and a set of graphs $\mathscr{F}$, let $ex(n, H, \mathscr{F})$ denote the maximum possible number of copies of $H$ in an $\mathscr{F}$-free graph on $n$ vertices. Solymosi and Wong \cite{SW17} proved that if the Erd\H{o}s's Girth Conjecture holds, then $ex(n, C_{2\ell}, \mathscr{C}_{2\ell-1}) = \Theta(n^{2\ell/(\ell-1)})$ for any $\ell\ge 3$, where $C_k$ is the cycle of length $k$ and  $\mathscr{C}_k = \{C_3, C_4, \dots , C_k\}$. Gerbner et al. \cite{GGMV20} prove that their result is sharp in the sense that forbidding any other even cycle decreases the number of $C_{2\ell}$'s significantly. Suprisingly, Theorem 1 implies that, $ex(2q^2, C_{6}, \mathscr{C}_{5}) \ge \frac{q^3(q-1)^2(q-2)}{6}$ and $ex(2q^3, C_{8}, \mathscr{C}_{7}) \ge \frac{q^4(q-1)^3(q-2)}{8}$ for $q$ odd, and the extremal graphs are Wenger graphs $W_1(q)$ and $W_2(q)$ respectively (Corollary \ref{ex(n,C,C)}).

\textbf{Theorem 2.} Let $q = p^e$ with $e \ge 1$ and $p$ a prime.
\begin{enumerate}
\item If $p$ is an odd prime and $m \ge 1$, then $L_m(q)$ is an $egr(2q^{m+1}, q,6,\lambda_m)$-graph with  $\lambda_1=(q-1)^2(q-2)$ and $\lambda_m=(q-1)^2(p-2)$ for $m\ge 2$.
\item If $p=2$, $e \ge 2$ and $m = 1$, then $L_1(q)$ is an $egr(2q^{2}, q,6,\lambda_1)$-graph with $\lambda_1=(q-1)^2(q-2)$.
\item If $p=2$, $e = m = 1$ or $e \ge 1$, $m \ge 2$, then $L_m(q)$ is an $egr(2q^{m+1}, q,8,\lambda_m)$-graph, where $\lambda_m=(q-1)^3+(q-1)^2(q-2)$.
\end{enumerate}

The paper is organized as follows:
Section \ref{ADG} recalls the definition and some properties of the algebraically defined graphs.  Section \ref{section def the some graphs} and Section \ref{Lm is egr} are devoted to proving our main theorems.
In the final part, we discuss the order of extremal $egr(v,q,g,(q-1)^{\frac{g-2}{2}}(q-2))$-graph for $g =8$.

%=========================================================
\section{Algebraically defined graphs} \label{ADG}
%\section{The girth cycles in Wenger graphs}\rm \label{section def the some graphs}
The present section mainly recall the definition of algebraically defined graphs. Throughout this paper,
we always denote by $V (\Gamma)$ the vertex set and by $E(\Gamma)$ the edge set for a graph $\Gamma$. The element of $E(\Gamma)$ joining two vertices $x, y \in V (\Gamma)$ will be written as $x\sim y$.

\begin{definition} \label{B Gamma n graph by equ}\rm \cite{LWA01}
Let $R$ be an arbitrary commutative ring with multiplicative identity, and $f_i : R^{2i-2} \rightarrow R $ be an arbitrary function, $2\leq i \leq n $. We define a bipartite graph
$B{\Gamma}_n = B{\Gamma}_n(R; f_2, \dots , f_n)$ as follows.
The set of vertices $V(B{\Gamma}_n)$ is the disjoint union of two copies of $R^n$, one is denoted by $P_n$ and the other by $L_n$.
Any element $p$ in $P_n$ is called a {\it point}, the element $l$ in $L_n$ is called a {\it line}, and denote by $(p) \in P_n$ and $[l] \in L_n$ respectively. Moreover,  a point
$(p) = (p_1, p_2, \dots , p_n)$ and a line $[l] = [l_1, l_2, \dots , l_n]$ are adjacent if and only if the
following $n-1$ relations on their coordinates hold:
\begin{align}
&p_2 + l_2 = f_2(p_1, l_1),         \nonumber \\
&p_3 + l_3 = f_3(p_1, l_1, p_2, l_2), \nonumber \\
&\;\;\;\;\;\; \dots   \;\;\;\; \dots                            \nonumber \\
&p_n + l_n = f_n(p_1, l_1, p_2, l_2, \dots , p_{n-1}, l_{n-1}). \nonumber
\end{align}
Note that in some cases, the construction has a slight difference with equations, for example, $a_ip_i+ b_il_i= f_i(p_1,l_1,\dots,p_{i-1},l_{i-1})$ for all $2 \le i \le n$ where $a_i$ and $b_i$ are fixed units in $R$. This family of graphs are also called  to be {\it algebraically defined graphs} (ADG for short).

Moreover, if a graph $\Gamma$ is isomorphic to $B\Gamma_n(R;f_1,\cdots,f_n)$ for some $R$ and $f_1, f_2, \cdots, f_n$, then $B\Gamma_n(R;f_1,\cdots,f_n)$ is called to be an {\it equation representation} of $\Gamma$.
\end{definition}
Evidently, if the order of $R$ is $r$, then the number of vertices of graph $B\Gamma_n$ is $2r^n$, and $B\Gamma_n$ is $r$-regular since for any fixed vertex $a \in V(B{\Gamma}_n)$, and an arbitrary element $x\in R$, there is a unique vertex $b\in V(B{\Gamma}_n)$ such that $b$ is adjacent to the vertex $a$ such that the first coordinate of $b$ is $x$.
%%The wenger graph $B\Gamma(n,q)$ is defined over a finite field and satisfies the following equation:

\begin{example}\rm
The Wenger graph $W_n(q)$ (cf. \cite{LFV02}) is defined over a finite field $\mathbb{F}_q$ and satisfies the following equations:
\begin{align}
p_i + l_i = p_1l_{i-1},\quad \text{where}\;\;i=2,\dots,n+1.         \nonumber
\end{align}
Moreover, the linearized Wenger graph $L_m(q)$ (see \cite{CXW15,WY17}) is defined by following equations:
\begin{align}
p_i + l_i = p_1^{p^{i-2}}l_1,\quad \text{where}\;\;i=2,\dots,m+1.         \nonumber
\end{align}
\end{example}
Generalized Kac-Moody algebras, as an important family of Lie algebras, are generated by some relations and a generalized Cartan matrix. The following construction of graphs corresponding to the Lie algebra $\Lie(M_1,3,q)$, $\Lie(M_2,4,q)$ and $\Lie(M_3,6,q)$ can be found in paper \cite{TAT12}. For the simplification of notations, we still denote by $\Lie(M_1,3,q)$, $\Lie(M_2,4,q)$ and $\Lie(M_3,6,q)$ the corresponding Lie graphs throughout the paper.

%\textbf{Construction of graph $\Lie(M_1,3,q)$.}
The graph $\Lie(M_1,3,q)$ is an ADG determined by one relation:
$$p_2+l_2=p_1l_1.$$

%\textbf{Construction of graph $\Lie(M_2,4,q)$.}
The graph $\Lie(M_2,4,q)$ is an ADG defined by the following two relations:
\begin{align}
p_2+l_2=p_1l_1,  \nonumber \\
p_3+l_3=p_1l_2. \nonumber
\end{align}

%\textbf{Construction of graph $\Lie(M_3,6,q)$.} %Since the basis of $\Lie(M_3,6,q)$ can be chosen to be  $$W'=\{e_1,e_2,[e_1,e_2],[e_2,[e_1,e_2]],[e_2,[e_2,[e_1,e_2]]],[e_1,[e_2,[e_2,[e_1,e_2]]]]\},$$
The graph $\Lie(M_3,6,q)$ is an ADG determined by the relations as follows.
\begin{align}
p_2+l_2&=p_1l_1,  \nonumber \\
p_3+l_3&=p_1l_2, \nonumber \\
p_4+l_4&=p_1l_3,   \nonumber \\
p_5+l_5&=p_2l_3-2p_3l_2+p_4l_1. \nonumber
\end{align}

\begin{remark} {\rm \begin{enumerate}
\setlength{\leftmargin}{2em}
\item By definition,  $\Lie(M_1,3,q)=W_1(q)$ and  $\Lie(M_2,4,q)=W_2(q)$ are Wenger graphs, but $\Lie(M_3,6,q)$ is not.
\item The Wenger graph $W_n(q)$ has different equation representations, for example, the graph $W_2(q)$ is isomorphic to $B\Gamma_2(\mathbb{F}_q;p_1l_1,p_1l_{1}^{2})$. In the following discussion, we will use an equation representation of $W_n(q)$ as $B\Gamma_{n+1}(\mathbb{F}_q;p_1l_1,\dots,p_1l_{1}^{i},\dots,p_1l_{1}^{n})$,  see paper \cite{CLL14} for details.
\end{enumerate}
}
\end{remark}

%%%%%%%%%%%%%%%%%%%%%%%%%%%%%%%%%%%%%%%%%%%

\section{ The edge-girth-regularity and the girth cycles in Wenger graphs}\rm \label{section def the some graphs}

An edge-girth-regular graph $egr(v,k,g,\lambda)$ is a $k$-regular graph of order $v$ and
girth $g$ satisfying that any edge is contained in $\lambda$ distinct $g$-cycles.
For example, the Petersen graph is an $egr(10,3,5,4)$-graph and a complete graph $K_n$ is an $egr(n,n-1,3,n-2)$-graph for $n\ge 3$.

%%%%%%%%%%%%%%%%%%%%%%%%%%%%%%%%%%%%%%%%%%%
Note that all edge-transitive graphs are edge-girth-regular graphs since any automorphism of graphs preserves cycles (see \cite{JKM18}).
Hence, the following two lemmas show that graphs $\Lie(M_3,6,q)$ and $W_n(q)$ are edge-girth-regular graph.
\begin{lemma}\cite[Theorem 3.1(c)]{LFU93}\label{L is et}
The graphs $\Lie(M_1,3,q)$, $\Lie(M_2,4,q)$ and $\Lie(M_3,6,q)$ are edge-transitive.
\end{lemma}

\begin{lemma}\cite[Theorem 2]{LFV02}\label{W is et}
The Wenger graph $W_n(q)$ is edge-transitive for all $n$ and $q$.
\end{lemma}

Note that the girth of $W_n(q)$ is $8$ for $n\ge 2$, and the girth of $W_1(q)$ is $6$ for $q\ge 3$, see \cite{LFU93,LTW18}.
Therefore, the order, degree, girth of $\Lie(M_1,3,q)$ and $\Lie(M_2,4,q)$ are already determined, and it suffices to describe the number of girth cycles with a certain edge in graphs $\Lie(M_1,3,q)$ and $\Lie(M_2,4,q)$, which is  completely described in the following theorem.

\begin{theorem}\label{f.d Lie graph egr}
Let $q\ge 3$ be a prime power. Then
\begin{enumerate}
\item $\Lie(M_1,3,q)$ is an $egr(2q^2, q,6,(q-1)^2(q-2))$-graph;
\item $\Lie(M_2,4,q)$ is an $egr(2q^3, q,8,(q-1)^3(q-2))$-graph for odd $q$;
\item $\Lie(M_2,4,q)$ is an $egr(2q^3, q,8,(q-1)^3(q-3)+2(q-1)^2)$-graph for even $q$.
\end{enumerate}
\end{theorem}
\begin{proof}
(1) By Lemma \ref{L is et}, the graph $\Lie(M_2,4,q)$ is edge-transitive, so it is sufficient to show that edge $(0,0)\sim [0,0]$ is contained in exactly $(q-1)^2(q-2)$ distinct $6$-cycles. By the adjacency relations of $\Lie(M_1,3,q)$, there is a closed walk $W_6$ as follows:
$$
(0,0)\sim[0,0]\sim(x,0)\sim[y,xy] \sim
(x',(x'-x)y)\sim[t,0]\sim(0,0),
$$
where $x,y,x',t \in \mathbb{F}_q$.
Since the girth $g(\Lie(M_2,4,q))=6$, $W_6$ is a cycle if and only if $0\neq x\neq x' \neq 0$ and $0\neq y\neq t \neq 0$.
Since the point $(x',(x'-x)y)$ and the line $[t,0]$ in the above closed walk are adjacent, we have $x't=(x'-x)y$.
It follows that $t$ has a solution $\frac{(x'-x)y}{x'}$ and is uniquely determined by $x$, $y$, $x'$.
Note that the number of possible values of $x$, $y$ and $x'$ are $q-1$, $q-1$ and $q-2$, respectively. So the number of different $6$-cycles containing edge $(0,0)\sim [0,0]$ is $(q-1)^2(q-2)$.

(2) By Lemma \ref{L is et}, the graph $\Lie(M_2,4,q)$ is edge-transitive, so we only need to show that edge $(0,0,0)\sim [0,0,0]$ is contained in exactly $(q-1)^3(q-2)$ distinct $8$-cycles. By the adjacency relations of $\Lie(M_2,4,q) \cong B\Gamma_3(\mathbb{F}_q;p_1l_1,p_1l_1^2))$, there is a closed walk $W_8$ as follows:
\begin{align}
&(0,0,0)\sim[0,0,0]\sim(x,0,0)\sim[y,xy,xy^2] \sim(x',(x'-x)y,(x'-x)y^2)\sim \nonumber\\
&[y',x'y'-(x'-x)y,x'y'^2-(x'-x)y^2]\sim(x'',x''t,x''t^2)\sim [t,0,0]\sim(0,0,0),\nonumber
\end{align}
where $x,y,x',y',x'',t \in \mathbb{F}_q$.
Since the point $(x'',x''t,x''t^2)$ and the line $[y',x'y'-(x'-x)y,x'y'^2-(x'-y)xy]$ in the above closed walk are adjacent, we have
\begin{align}
x''y'&=x''t+x'y'-(x'-x)y \nonumber\\
x''y'^2&=x''t^2+x'y'^2-(x'-x)y^2 \nonumber
\end{align}
It is clear that the above equations are equivalent to the following equations.
\begin{align}
x''(y'-t)&=x'y'-(x'-x)y   \label{eq1 of 8-cycle}\\
x''(y'-t)(y'+t)&=x'y'^2-(x'-x)y^2 \label{eq2 of 8-cycle}
\end{align}
Since the girth $g(\Lie(M_2,4,q))=8$, $W_8$ is a cycle if and only if $0\neq x\neq x' \neq x'' \neq 0$ and $0\neq y\neq y' \neq t \neq 0$. Then $t$ and $x''$ have solutions and are uniquely determined by the values of $x$, $y$, $x'$ and $y'$.

%\textcolor{red}{Now, we count the number of possible values of $x$, $y$, $x'$ and $y'$, which are $q-1$, $q-1$, $q-1$ and $q-2$, respectively. Hence, the edge $(0,0,0)\sim [0,0,0]$ is contained in exactly $(q-1)^3(q-2)$ distinct $8$-cycles.}

Now we count the number of different $8$-cycles containing edge $(0,0,0)\sim [0,0,0]$. Note that  $y=t$ if and only if $y'=0$, by equations \eqref{eq1 of 8-cycle} and \eqref{eq2 of 8-cycle}. We divide the discussion into the following four cases.

{\bf Case 1.} $x'=0$ and $y'=0$. By equations \eqref{eq1 of 8-cycle} and \eqref{eq2 of 8-cycle}, we get $t=y$ and $x''=-x$.
The number of possible values of $x$, $y$ and $y'$ are $q-1$, $q-1$ and $1$, respectively. So the number of different $8$-cycles containing edge $(0,0,0)\sim [0,0,0]$ in this case is $(q-1)^2$.

{\bf Case 2.} $x'=0$ and $y'\neq0$. By equations \eqref{eq1 of 8-cycle} and \eqref{eq2 of 8-cycle}, it follows that $t=y-y'$ and $x''(2y'-y)=xy$.
%If $y=0$, then $x''=0$ or $y'=t$, which is impossible.
Since $y\neq 0$ and $x\neq 0$, we have $y' \neq y, \frac{y}{2}, 0$ and $y\neq \frac{y}{2} \neq 0$. The number of possible values of $x$, $y$ and $y'$ are $q-1$, $q-1$ and $q-3$, respectively. So the number of different $8$-cycles containing edge $(0,0,0)\sim [0,0,0]$ is $(q-1)^2(q-3)$.

{\bf Case 3.} $x'\neq 0$ and $y'=0$. Then $y=t$ and $x''=x'-x$, we count the number of possible values of $x$, $y$, $x'$ and $y'$, which are $q-1$, $q-1$, $q-2$ and $1$, respectively. Then the number of different $8$-cycles containing edge $(0,0,0)\sim [0,0,0]$ is $(q-1)^2(q-2)$.

{\bf Case 4.} $x'\neq 0$ and $y'\neq 0$. Then $t\neq y$. By equations \eqref{eq1 of 8-cycle} and \eqref{eq2 of 8-cycle},
$$\left(x'y'-(x'-x)y\right)(y'+t)=x'y'^2-(x'-x)y^2.$$
Set $a=x'y'-(x'-x)y$ and $b=\frac{x'-x}{x'}$. Then $y'\neq by$ since $a\neq 0$. It follows that $t$ and $x''$ have solutions and are uniquely determined by $x$, $y$, $x'$ and $y'$. Moreover, $b\neq 1,0$ since $x\neq 0$ and $x'\neq x$. Then $y'\neq y,by,0$ and $0\neq y\neq by \neq 0$. Now, we count the number of possible values of $x$, $y$, $x'$ and $y'$, which are $q-1$, $q-1$, $q-2$ and $q-3$, respectively. Then the number of different $8$-cycles containing edge $(0,0,0)\sim [0,0,0]$ is $(q-1)^2(q-2)(q-3)$ in this case.

By the above four cases, there are exactly $(q-1)^2+(q-1)^2(q-3)+(q-1)^2(q-2)^2=(q-1)^3(q-2)$ $8$-cycles containing edge $(0,0,0)\sim [0,0,0]$.

(3) Suppose that $q=2^e$ and $e>1$. The counting is similar to the discussion in the proof of (2) as above with slight differences in Case 2 and Case 4. More precisely, for Case 2 ($x'=0$ and $y'\neq0$),  it follows that $x''=x$ and $y'\neq y,0$ since $2y'=0$, and the number of different $8$-cycles containing edge $(0,0,0)\sim [0,0,0]$ becomes $(q-1)^2(q-2)$.  While for the case 4 ($x'\neq 0$ and $y'\neq 0$), $x'y'^2-(x'-x)y^2\neq 0$ since $y'\neq -t=t$.
Thus,
$$y'^2\neq by^2=(b^{2^{e-1}}y)^2 \quad \text{and} \quad y'\neq b^{2^{e-1}}y.$$
Note that $yb\neq b^{2^{e-1}}y$. Then the number of different $8$-cycles containing edge $(0,0,0)\sim [0,0,0]$ is $(q-1)^2(q-2)(q-4)$. So the number of $8$-cycles containing edge $(0,0,0)\sim [0,0,0]$ is exactly $(q-1)^2+(q-1)^2(q-2)+(q-1)^2(q-2)+(q-1)^2(q-2)(q-4)=(q-1)^3(q-3)+2(q-1)^2$.
\end{proof}

Now we are ready to prove our first main theorem.
%\begin{theorem}\label{n of cycles in W}The Wenger graph $W_n(q)$ is an $egr(2q^{n+1}, q,8,\lambda_n)$-graph for $n\ge 2$, where $\lambda_2 = (q-1)^3(q-2)$ for $q$ odd, $\lambda_2 = (q-1)^3(q-3)+2(q-1)^2$ for $q$ even, and $\lambda_n = (q-1)^3$ for $n\ge 3$. Moreover, $W_1(q)$ is an $egr(2q^2, q,6,(q-1)^2(q-2))$-graph for $q\ge 3$, $W_1(2)$ is an $egr(8, 2,8,1)$-graph.
%\end{theorem}
\begin{theorem}\label{n of cycles in W}
Let $q$ be a prime power and $n$ be a positive integer. Then Wenger graph $W_n(q)$ is an $egr(2q^{n+1}, q,g_n,\lambda_n)$-graph for all $n$ and $q$. More precisely,
\begin{enumerate}
\item $W_1(q)$ is an $egr(2q^2, q,6,(q-1)^2(q-2))$-graph for $q\ge 3$, and $W_1(2)$ is an $egr(8, 2,8,1)$-graph;
\item the girth $g_2$ of $W_2(q)$ is $8$, $\lambda_2 = (q-1)^3(q-2)$ for $q$ odd, and $\lambda_2 = (q-1)^3(q-3)+2(q-1)^2$ for $q$ even;
\item for any $q$, the girth $g_n$ of $W_n(q)$ is $8$ and $\lambda_n = (q-1)^3$ when $n\ge 3$.
\end{enumerate}
\end{theorem}
\begin{proof}
The Wenger graph $W_n(q)$ is edge-girth-regular graph by Lemma \ref{W is et}. Moreover, the girth of $W_n(q)$ is $8$ for $n\ge 2$, and the girth of $W_1(q)$ is $6$ for $q\ge 3$, see \cite{LFU93,LTW18}. By a straightforward check, $W_n(2)$ is an $egr(2^{n+2}, 2,8,1)$-graph since $W_n(2)$ is $2$-regular for all $n\ge 1$. Since the Wenger graphs $W_1(q)=\Lie(M_1,3,q)$ and  $W_2(q)=\Lie(M_2,4,q)$, the first statement follows directly by Theorem \ref{f.d Lie graph egr}.

For the rest two statements, it suffices to prove that $\lambda_n =(q-1)^3$ for $n\ge 3$. We consider an $8$-cycle of the following form
\begin{align}
&(0,0,0,\ldots)\sim[0,0,0,\ldots]\sim(x,0,0,\ldots)\sim[y,xy,xy^2,\ldots] \sim(x',(x'-x)y,(x'-x)y^2,\ldots)\sim \nonumber\\
&[y',x'y'-(x'-x)y,x'y'^2-(x'-x)y^2,\ldots]\sim(x'',x''t,x''t^2,\ldots)\sim [t,0,0,\ldots]\sim(0,0,0,\ldots).\nonumber
\end{align}
Then $x$, $y$, $x'$ $y'$, $x''$ and $t$ satisfy equations \eqref{eq1 of 8-cycle}, \eqref{eq2 of 8-cycle} and the following equations.
\begin{align}
x''y'^k=x''t^k+x'y'^k-(x'-x)y^k \quad \text{for} \quad 3\le k \le n. \label{eq3 of 8-cycle}
\end{align}

{\bf Case 1.} $y'=0$. Note that $y=t$ if and only if $y'=0$. Then $t=y$ and $x''=x'-x$ satisfy equations \eqref{eq1 of 8-cycle}, \eqref{eq2 of 8-cycle} and \eqref{eq3 of 8-cycle}. Now, we count the number of possible values of $x$, $y$, $x'$ and $y'$, which are $q-1$, $q-1$, $q-1$ and $1$, respectively. Then the number of different $8$-cycles containing edge $(0,0,0)\sim [0,0,0]$ is $(q-1)^3$.

{\bf Case 2.} $y'\neq 0$. We obtain $y\neq t$ and
	\begin{equation}\nonumber
		\left[ \begin{array}{ccc}
			y'^m & y^m & t^m \\
			y'^{m+1} & y^{m+1} & t^{m+1} \\
			y'^{m+2} & y^{m+2} & t^{m+2}
		\end{array}
		\right ]
		\left[ \begin{array}{ccc}
			x''-x' \\
			x'-x \\
			-x''
		\end{array}
		\right ]
		=
		\left[ \begin{array}{ccc}
			0 \\
			0 \\
			0
		\end{array}
		\right ],\;
	\end{equation}
where $m\le n-2$ is a positive integer. Since $0\neq y\neq y'$, $y'\neq t \neq 0$, and $t \neq y$, the matrix corresponding to the above linear equations is an invertible matrix, we obtain  $x''-x'=0$, which is impossible.

By the above two cases, the number of $8$-cycles containing edge $(0,\dots,0)\sim [0,\dots,0]$ in $W_n(q)$ is exactly $(q-1)^3$ .
Therefore, $W_n(q)$ is an $egr(2q^{n+1}, q,8,\lambda_n)$-graph with $\lambda_n = (q-1)^3$ for $n\ge 3$.
\end{proof}

The girth cycle of the algebraic bipartite graph $D(k,q)$ is characterized by equations in \cite{XCT23}, but the authors stated that the number of girth cycles in the algebraic bipartite graph $D(k,q)$ is still not transparent, even when $k=3$. Note that $D(3,q)=W_2(q)$.

Let $\Gamma$ be an edge-girth-regular graph $egr(v,k,g,\lambda)$. By Proposition 2.3 (iii) in \cite{JKM18}, the number of girth cycles of graph $\Gamma$ is $\frac{vk\lambda}{2g}$. Therefore, determining the parameter $\lambda$ is equivalent to determining the number of girth cycles of graph $\Gamma$.
As a corollary, we can obtain the number of girth cycles of Wenger graph for the case that $n=2$ and $q$ odd. In particular, the number of girth cycles of graph $D(3,q)=W_2(q)$ is determined. It is riveting that the number of girth cycles of $W_2(3)$ (the Gray graph) is $81$.

%\modifiedtime{2023-6-29 11:10:45  Left to F-Y Yang, why this single case? Other cases are hard to compute?}

%\begin{corollary} \label{the number of girth cycles of W_2(q)}
%The number of girth cycles in $W_2(q)$ is equal to $\frac{vk\lambda}{2g}=\frac{q^4(q-1)^3(q-2)}{8}$ for $q$ odd.
%\end{corollary}

\begin{corollary} \label{the number of girth cycles of W_2(q)}
Let $q = p^e$ with $p$ an odd prime. Then
\begin{enumerate}
\item the number of girth cycles in $W_1(q)$ is equal to $\frac{vk\lambda}{2g}=\frac{q^3(q-1)^2(q-2)}{6}$;
\item the number of girth cycles in $W_2(q)$ is equal to $\frac{vk\lambda}{2g}=\frac{q^4(q-1)^3(q-2)}{8}$.
\end{enumerate}
\end{corollary}
We recall the generalized Tur\'an number $ex(n, H, \mathscr{F})$. Given a graph $H$ and a set of graphs $\mathscr{F}$, let $ex(n, H, \mathscr{F})$ denote the maximum possible number
of copies of $H$ in an $\mathscr{F}$-free graph on $n$ vertices. In particular, let $\mathscr{C}_{k}=\{C_3,C_4,\dots,C_k\}$, by the above corollary,
we will show that Wenger graphs give the tight bounds
on the generalized Tur\'an numbers $ex(2q^2, C_{6}, \mathscr{C}_{5})$ and $ex(2q^3, C_{8}, \mathscr{C}_{7})$.
%we have the lower bounds of $ex(2q^2, C_{6}, \mathscr{C}_{5})$ and $ex(2q^3, C_{8}, \mathscr{C}_{7})$ for $q$ odd. %since $W_2(q)$ is a $\mathscr{C}_{7}$-free graph,
%Moreover, the extremal graphs are Wenger graphs $W_1(q)$ and $W_2(q)$, respectively.}
%\begin{corollary} \label{ex(n,C,C)}
%The generalized Tur\'an number $ex(2q^3, C_{8}, \mathscr{C}_{7}) \ge \frac{q^4(q-1)^3(q-2)}{8}=\Theta\left((2q^3)^{8/3}\right)$ for $q$ odd.
%\end{corollary}
\begin{corollary} \label{ex(n,C,C)}
Let $q = p^e$ with $p$ an odd prime. We have
\begin{enumerate}
\item $ex(2q^2, C_{6}, \mathscr{C}_{5}) \ge \frac{q^3(q-1)^2(q-2)}{6}$, and $ex(n,C_{6}, \mathscr{C}_{5})\ge\left(\frac{1}{48}-o(1)\right)n^3$;
\item $ex(2q^3, C_{8}, \mathscr{C}_{7}) \ge \frac{q^4(q-1)^3(q-2)}{8}$, and $ex(n,C_{8}, \mathscr{C}_{7})\ge \left(\frac{1}{2^{17/3}}-o(1)\right)n^{8/3}$.
\end{enumerate}
\end{corollary}
\begin{proof}
By the prime number theorem, then the largest prime below $n$ has size $n-o(n)$. If $p$ is the largest prime such that $2p^2\le n$, then $p=\left(\frac{1}{2^{1/2}}-o(1)\right)n^{1/2}$. Hence,
$$ex(n,C_{6}, \mathscr{C}_{5})\ge ex(2p^2, C_{6}, \mathscr{C}_{5}) \ge \frac{p^3(p-1)^2(p-2)}{6}=\left(\frac{1}{6}-o(1)\right)p^6=\left(\frac{1}{48}-o(1)\right)n^3.$$
Similarly,
$$ex(n,C_{8}, \mathscr{C}_{7})\ge ex(2p^3, C_{8}, \mathscr{C}_{7}) \ge \frac{p^4(p-1)^3(p-2)}{8}=\left(\frac{1}{8}-o(1)\right)p^8=\left(\frac{1}{2^{17/3}}-o(1)\right)n^{8/3}.$$
\end{proof}
\begin{remark}\rm
Note that the lower bounds in the above corollary are larger than the lower bounds in \cite[Proposition 8]{SW17}.
Let $G$ be a bipartite $n$-vertex graph having girth larger than $2(\ell-1)$ and $0<c_1n^{1/(\ell-1)}\leq\deg_G(v)\leq{c_2}n^{1/(\ell-1)}$ for any vertex $v\in V(G)$.
 By Proposition 8 in \cite{SW17}, the generalized Tur\'an number $ex(n,C_{2\ell}, \mathscr{C}_{2\ell-1})\ge \alpha_{\ell}{n^{2\ell/(\ell-1)}}$, where
$$\alpha_\ell=\dfrac{9}{4\ell{c_2}}\left(c_1^{\ell}/10\right)^{3}>0.$$
Let $n$ be sufficiently large. The Tur\'an numbers $ex(n, C_4) \le 0.5n^{3/2}$ and $ex(n, C_6) < 0.6272n^{4/3}$ appear in \cite{B1966} and \cite{FNV06}, respectively.
Then $c_1 \le 0.6272$ for $\ell=3$ or $4$.
By a straightforward check, $\alpha_\ell<\frac{1}{2^{17/3}}<\frac{1}{48}$ for $\ell=3$ or $4$.
\end{remark}
%Girth cycles with a certain edge in \texorpdfstring{$L_m(q)$}{}
\section{The edge-girth-regularity of \texorpdfstring{$L_m(q)$}{}}\label{Lm is egr}
In this section, we mainly prove that the linearized Wenger graph $L_m(q)$ is an $egr(v,q,g,\lambda)$-graph and determine the number of cycles of length $g$ passing through any edge in  $L_m(q)$ of girth $g$.

Before obtaining this parameter $\lambda$ of $L_m(q)$, we list several properties of the linearized Wenger graph. The following one determines the girth of the linearized Wenger graph $L_m(q)$.

\begin{lemma}\cite{CXW15}\label{girth of Lm} Let $L_m(q)$ be the linearized Wenger graph. Then

{\rm (1)} if $q = p^e$ with either $m \ge 1$, $e \ge 1$ and $p$ is an odd prime, or $m = 1$, $e \ge 2$ and $p = 2$, then the girth of $L_m(q)$ is $6$;

{\rm (2)} if $q = p^e$ with either $p = 2$ and $e = m = 1$ or $p = 2$, $e \ge 1$ and $m \ge 2$, then the girth of $L_m(q)$ is $8$.
\end{lemma}

The following lemma implies that, similar to the Wenger graphs, the linearized Wenger graphs are edge-transitive.

\begin{lemma}\label{Lm is edge-trans}
The linearized Wenger graph $L_m(q)$ is edge-transitive for all $q$ and $m$.
\end{lemma}
\begin{proof}
For any $0\le i \le m+1$ and $x\in \mathbb{F}_q$, we define the assignments $\sigma_{i,x} : L_m(q) \rightarrow L_m(q) $ as follows.
	\begin{align}
		\sigma_{1,x} \;:\; &(p_1,\cdots,p_{m+1}) \mapsto (p_1,p_2+p_1x,p_3+p_1^{p}x,\cdots,p_{m+1}+p_1^{p^{m-1}}x),  \nonumber \\
		&[\;l_1,\cdots,l_{m+1}\;] \mapsto [\;l_1+x,l_2,l_3,\cdots,l_{m},l_{m+1}\;].     \nonumber \\
	 \sigma_{i,x}\; : \; &(p) \mapsto (p_1,\cdots,p_{i-1},p_i-x,p_{i+1},p_{i+2},\cdots,p_{m+1}), \nonumber \\
		&[\;l\;] \mapsto [\;l_1,\cdots,l_{i-1},l_i+x,l_{i+1},l_{i+2},\cdots,l_{m+1}\;],  \; \text{ for } m+1\ge i\ge 2. \nonumber\\
		\sigma_{0,x}:\;  &(p) \mapsto (p_1+x, p_2,p_3,\cdots,p_{m},p_{m+1}), \nonumber \\
		&[\;l\;] \mapsto [\;l_1,l_2+l_1x,l_3+l_1x^{p},\cdots,l_{m+1}+l_1x^{p^{m-1}}\;].  \nonumber
\end{align}
By a straightforward check, $\sigma_{i,x}$ is the automorphism of graph $L_m(q)$.
	
	Let $(p)\sim [l]$ be an edge in the graph $L_m(q)$.
Suppose that $\sigma_1=\prod \limits_{i=1}^{m+1}\sigma_{i,-l_i}$. It is straightforward to check that $\sigma_1([l])=[0]$. Suppose that $(p')=\sigma_1((p))$, then $\sigma_1((p)\sim[l])=(p')\sim[0]$. Notice that $(p')\sim[0]$ is an edge in $L_m(q)$, then $p'_i=0$ for all $i=2,\cdots,m+1$. Set $\sigma=(\sigma_{0,-p'_1})\circ \sigma_1$. Then we have $\sigma((p)\sim[l])=\sigma_{0,-p'_1}((p')\sim[0])=(0)\sim[0]$. Since any edge in the graph $L_m(q)$ can be transformed to the edge $(0)\sim[0]$ by some automorphism, $L_m(q)$ is edge-transitive.
\end{proof}

The following theorem determines the parameter $\lambda$ of the linearized Wenger graphs.
\begin{theorem}
Let $q = p^e$ with $e \ge 1$ and $p$ a prime.
\begin{enumerate}
	\item If $p$ is an odd prime and $m \ge 1$, then $L_m(q)$ is an $egr(2q^{m+1}, q,6,\lambda_m)$-graph with  $\lambda_1=(q-1)^2(q-2)$ and $\lambda_m=(q-1)^2(p-2)$ for $m\ge 2$.
	\item If $p=2$, $e \ge 2$ and $m = 1$, then $L_1(q)$ is an $egr(2q^{2}, q,6,\lambda_1)$-graph with $\lambda_1=(q-1)^2(q-2)$.
	\item If $p=2$, $e = m = 1$ or $e \ge 1$, $m \ge 2$, then $L_m(q)$ is an $egr(2q^{m+1}, q,8,\lambda_m)$-graph, where $\lambda_m=(q-1)^3+(q-1)^2(q-2)$.
\end{enumerate}
\end{theorem}
\begin{proof}
The girth of the linearized Wenger graph $L_m(q)$ is determined in Lemma \ref{girth of Lm}.
By Lemma \ref{Lm is edge-trans}, the graph $L_m(q)$ is edge-transitive, and thus edge-girth-regular.

For the statement (1) and (2), we only need to show that edge $(0,\ldots,0)\sim [0,\ldots,0]$ is contained in exactly $\lambda_m$ distinct $6$-cycles. By the adjacency relations of $L_m(q)$, there is a cycle $C_6$ of the following form
\begin{align}
&(0,\ldots,0)\sim[0,\ldots,0]\sim(x,0,\ldots,0)\sim[y,xy,x^{p}y,\ldots,x^{p^{m-1}}y] \sim \nonumber\\
&(z,zy-xy,z^{p}y-x^{p}y,\ldots,z^{p^{m-1}}y-x^{p^{m-1}}y)\sim[t,0,\ldots,0]\sim(0,\ldots,0),\nonumber
\end{align}
where $0\neq x \neq z \neq 0$ and $0\neq y \neq t \neq 0$. Then $z^{p^i}y-x^{p^i}y=z^{p^i}t$ for all $0\le i \le m-1$.

For $m=1$, the variable $t$ has a solution such that $t\neq 0,y$ since $z\neq 0, x$. Then the number of possible choices of $x$, $y$, $z$ are $q-1$, $q-1$, $q-2$, respectively. So the number of different $6$-cycles containing edge $(0,0)\sim [0,0]$ is $(q-1)^2(q-2)$.

For the case $m\ge 2$, the variable $t$ has a solution if and only if
$$z^{p^j-p^{j-1}}=x^{p^j-p^{j-1}},$$
for all $2\le j \le m$. Note that $z\neq 0 \neq x$.
Let $X=\frac{z}{x}$. So
$$X^{p^j}-X^{p^{j-1}}=0=(X^p-X)^{p^{j-1}}\;\text{if and only if}\;X^p=X,$$
for all $2\le j \le m$. The equation $X^p=X$ has exactly $p-1$ different non-zero solutions in $\mathbb{F}_q$. But $X=1$ forces $L_m(q)$ to have a cycle with a length less than $6$, which is impossible.
Therefore, if we fix $x$, then $z$ has exactly $p-2$ choices.

Now, we count the number of possible values of $x$, $y$, $z$, which are $q-1$, $q-1$, $p-2$, respectively. Then the number of different $6$-cycles containing edge $(0,\dots,0)\sim [0,\dots,0]$ is $(q-1)^2(p-2)$.

To prove the statement (3), it suffice to show that edge $(0,\ldots,0)\sim [0,\ldots,0]$ is contained in exactly $(q-1)^3+(q-1)^2(q-2)$ distinct $8$-cycles. By the adjacency relations of $L_m(q)$, there is a cycle $C_8$ as follows:
\begin{align}
&(0,\ldots,0)\sim[0,\ldots,0]\sim(x,0,\ldots,0)\sim[y,xy,x^{p}y,\ldots,x^{p^{m-1}}y]  \nonumber\\
&\sim (z,zy-xy,z^{p}y-x^{p}y,\ldots,z^{p^{m-1}}y-x^{p^{m-1}}y)  \nonumber\\
&\sim [u,zu-(zy-xy),\ldots,z^{p^{m-1}}u-z^{p^{m-1}}y-x^{p^{m-1}}y]  \nonumber\\
&\sim (v,vt,\ldots,v^{p^{m-1}}t)\sim [t,0,\ldots,0]\sim(0,\ldots,0),\nonumber
\end{align}
where $0\neq x\neq z \neq v \neq 0$ and $0\neq y\neq u \neq t \neq 0$.
Since $p=2$, we have the following equations $E_k$.
$$v^{2^{k-1}}t+\left(z^{2^{k-1}}-x^{2^{k-1}}\right)y=\left(z^{2^{k-1}}+v^{2^{k-1}}\right)u\;\;\text{for}\;\;1\le k \le m.$$
By equations $E_1$ and $E_2$, we note that $u=0$ if and only if $v=z-x$ and $y=t$.

{\bf Case 1.} $u=0$. Then $v=z-x$ and $y=t$ and they satisfy equations $E_k$ for $1\le k \le m$. Thus, there are exactly $(q-1)^3$ $8$-cycles containing edge $(0,\dots,0)\sim [0,\dots,0]$ in $L_m(q)$.

{\bf Case 2.} $u\neq 0$ and $x=v$. Then $v\neq z-x$ or $y\neq t$. If $x=v$, then $z=0$ and $t=u-y$ . These solutions satisfy equations $E_k$ for $1\le k \le m$. Now, we count the number of possible choices of $x$, $y$, $z$ and $u$, which are $q-1$, $q-1$, $1$ and $q-2$, respectively. Then the number of different $8$-cycles containing edge $(0,0,0)\sim [0,0,0]$ is $(q-1)^2(q-2)$.

{\bf Case 3.} $u\neq 0$ and $x\neq v$. Note that $y\neq t$ if $v=z-x$. By equations $E_1$ and $E_2$, $v=z+v$, $z=0$, and $v=x$ when $v=z-x$, which is impossible. For the case that $v\neq z-x$, we have
\begin{equation}\nonumber
		\left[ \begin{array}{ccc}
			v & z & x \\
			v^2 & z^2 & x^2 \\
			v^4 & z^4 & x^4
		\end{array}
		\right ]
		\left[ \begin{array}{ccc}
			t-u \\
			u-y \\
			y
		\end{array}
		\right ]
		=
		\left[ \begin{array}{ccc}
			0 \\
			0 \\
			0
		\end{array}
		\right ].\;
	\end{equation}
We notice that the determinant of the coefficient matrix above is $vzx(z-v)(x-v)(z-x)(x-z+v)\neq 0$ if $v\neq z-x$. Thus $y=0$, which is impossible.

By the discussion as in above three cases, we have $\lambda_m=(q-1)^3+(q-1)^2(q-2)$.
\end{proof}

\section{The order of extremal Edge-girth-regular graphs}\label{extremalegrg}
In this section, we first recall the notation of extremal edge-girth-regular graphs, which is a special class of edge-girth-regular graphs orginally introduced by Drglin et al. These graphs are $egr(v,k,g,\lambda)$-graphs with $v=n(k,g,\lambda)$, where $n(k,g,\lambda)$ is the smallest value of $w$ such that there exists at least one $egr(w,k,g,\lambda)$-graph for a fixed triplet $(k,g,\lambda)$.

The lower bound of the order of any extremal $egr(v,k,g,\lambda)$-graph is at least the Moore bound, where the Moore bound $M(k,g)$ defined as follows
$$M(k,g) = \left\{ \begin{array}{ll} 1+ k+k(k-1)+\dots+k(k-1)^{\frac{g-3}{2}} , &\mbox{ if $g$ is odd};\nonumber \\
 2\left(1+(k-1)+\dots+(k-1)^{\frac{g-2}{2}}\right), &\mbox{ if $g$ is
even}.\end{array}\right.
$$

We list some results about the order of extremal edge-girth-regular graphs in \cite{DFJR21}.
\begin{theorem}\cite{DFJR21}\label{extremalbound}
Let $k$ and $g$ be a fixed pair of integers greater than or equal to $3$, and let $\lambda\leq (k-1)^{\frac{g-1}{2}}$ when $g$ is odd, and $\lambda \leq (k-1)^{\frac{g}{2}}$ when $g$ is even. Then

\begin{equation}\label{lowercages} n(k,g,\lambda)  \geq M(k,g) + \left\{ \begin{array}{ll} (k-1)^{\frac{g-1}{2}}- \lambda , &\mbox{ if $g$ is odd};\nonumber \\
 \lceil 2 \frac{(k-1)^{\frac{g}{2}}- \lambda}{k}\rceil, &\mbox{ if $g$ is
even}.\end{array}\right.\end{equation}

\end{theorem}

Moreover, the authors \cite{DFJR21} improved the lower bound for bipartite graphs and gave the following theorem.

\begin{theorem}\cite{DFJR21}\label{extremalbipartitebound}
Let $k\geq 3$ and $g\geq 4$ be even, and let $\lambda\leq (k-1)^{\frac{g}{2}}$. If $\mathcal G$ is a bipartite $egr(v,k,g,\lambda)$-graph, then
$$v  \geq M(k,g) + 2 \lceil  \frac{(k-1)^{\frac{g}{2}}- \lambda}{k}\rceil.$$
\end{theorem}

Now we estimate the order of extremal $egr(v,q,g,(q-1)^{\frac{g-2}{2}}(q-2))$-graph for $g \in \{6,8\}$. We deduce a corollary as follows.

\begin{corollary}
Let $q$ be an odd prime power and $egr(v,q,g,(q-1)^{\frac{g-2}{2}}(q-2))$-graph be an extremal edge-girth-regular graph. Then for $g \in \{6,8\}$, we have
\[(1-o(1))\left(2q^{\frac{g-2}{2}}\right)\leq v\leq  2q^{\frac{g-2}{2}}.\]
%$v = (1-o(1))\left(2q^{\frac{g-2}{2}}\right)$ for $g \in \{6,8\}$.
\end{corollary}
\begin{proof}
On one hand, by Theorem \ref{extremalbipartitebound},
$$v \geq M(q,g) + 2 \lceil  \frac{(q-1)^{\frac{g}{2}}- \lambda}{q}\rceil=(1-o(1))\left(2q^{\frac{g-2}{2}}\right).$$
On the other hand, by Theorem \ref{n of cycles in W}, for $g \in \{6,8\}$, there are graph families $egr(v,q,g,(q-1)^{\frac{g-2}{2}}(q-2))$ which are $\Lie(M_1,3,q)$ and $\Lie(M_2,4,q)$, respectively. It implies $v\le 2q^{\frac{g-2}{2}}$ for $g \in \{6,8\}$.
%$$v = (1-o(1))\left(2q^{\frac{g-2}{2}}\right)$$ for $g \in \{6,8\}$.
\end{proof}

By Lemma \ref{L is et}, the Lie graph $\Lie(M_3,6,q)$ is an $egr(2q^5,q,g,\lambda)$-graph, and the girth of $\Lie(M_3,6,q)$ is $12$ for $q$ not powers of $2$ or $3$. Thus, the following problem seems to be crucial to consider the case of $g=12$ in Conjecture \ref{conj 1}.
\begin{problem}
How to determine the value of parameter $\lambda$ of the Lie graph $\Lie(M_3,6,q)$ and the graph $D(k,q)$ for $k\ge 4$?
\end{problem}
%=========================================================
\section*{acknowledgements}
%======================================================
This work is supported by National Natural Science Foundation of China (Grant Nos. 11801494, 11961007, 11971196, 12061060).

\end{document}